\documentclass[12pt]{article}

\usepackage{amsmath, amssymb, amsthm, mathrsfs, eucal}

\usepackage[T1]{fontenc}

\newtheorem{theorem}{Theorem}
\newtheorem{corollary}[theorem]{Corollary}
\newtheorem{lemma}[theorem]{Lemma} 
\newtheorem{proposition}[theorem]{Proposition}

\theoremstyle{plain}

\theoremstyle{definition}
\newtheorem*{xrem}{Remark}

\def\mod{\text{mod }}

\newcommand{\GQ}{\operatorname{G}_\QQ}
\newcommand{\Gal}{\operatorname{Gal}(\overline{\mathbb Q}/\mathbb Q)}
\newcommand{\ov}{\overline}
\newcommand{\ZZ}{\mathbb{Z}}
\newcommand{\NN}{\mathbb{N}}
\newcommand{\QQ}{\mathbb{Q}}
\newcommand{\FF}{\mathbb{F}}

\begin{document}

\title{On a class of generalized Fermat equations of signature $(2,2n,3)$} 

\author{Karolina Cha\l upka,  Andrzej D\k{a}browski  and G\"okhan Soydan} 

\date{\today}

\maketitle

{\it Abstract}. 
We consider the Diophantine equation $7x^{2} + y^{2n} = 4z^{3}$. We determine all
solutions to this equation for $n = 2, 3, 4$ and $5$. We formulate a Kraus type criterion for
showing that the Diophantine equation $7x^{2} + y^{2p} = 4z^{3}$ has no non-trivial
proper integer solutions for specific primes $p > 7$. We computationally verify
the criterion for all primes $7 < p < 10^9$, $p \neq 13$. We use the symplectic
method and quadratic reciprocity to show that the Diophantine equation
$7x^{2} + y^{2p} = 4z^{3}$ has no non-trivial proper solutions for a positive proportion
of primes $p$. In the paper \cite{ChDS} we consider the Diophantine equation $x^{2} +7y^{2n} = 4z^{3}$, 
determining all families of solutions for $n=2$ and $3$, as well as giving a (mostly) conjectural description 
of the solutions for $n=4$ and primes $n \geq 5$.

\bigskip


\bigskip


\section{Introduction}

Fix nonzero integers $A$, $B$ and $C$.   For given positive integers $p$, $q$, $r$ 
satisfying $1/p + 1/q + 1/r < 1$, the generalized Fermat equation 
\begin{equation} \label{basicFermat}  
Ax^p + By^q = Cz^r 
\end{equation} 
has only finitely many primitive integer solutions.  
Modern techniques coming from Galois representations and modular forms 
(methods of Frey--Helle-gouarch curves and variants of Ribet's level-lowering 
theorem, and of course, the modularity of elliptic curves or abelian varieties over the rationals 
or totally real number fields) allow to give partial (sometimes complete) 
results concerning the set of solutions to \eqref{basicFermat} (usually, when  
a radical of $ABC$ is small), at least when 
$(p,q,r)$ is of the type $(n,n,n)$, $(n,n,2)$, $(n,n,3)$,  $(2n,2n,5)$, $(2,4,n)$, 
$(2,6,n)$, $(2,n,4)$,  $(2,n,6)$, $(3,3,p)$,  $(2,2n,3)$, $(2,2n,5)$.  
Recent papers by Bennett, Chen, Dahmen and Yazdani \cite{BCDY} 
and by Bennett, Mih\v{a}ilescu and Siksek \cite{BMS} survey approaches to solving 
the equation  \eqref{basicFermat} when $ABC=1$. 
When $1/p + 1/q + 1/r$ is close to one (for instance, for $(p,q,r) = (2,3,5)$, $(2,3,7)$ or $(2,3,8)$), 
then one needs new methods (Chabauty method or its refinements \cite{Br0,Br,Br2}, 
or a combination of Chabauty type method with a modular approach  \cite{FNS,PSS}).

The Diophantine equation 
\begin{equation} \label{Chen-Dahmen} 
x^2+y^{2n}=z^3 
\end{equation} 
was studied by Bruin \cite{Br}, Chen \cite{Ch},   Dahmen \cite{Da},   and  Bennett, Chen, Dahmen 
and Yazdani \cite{BCDY}.  It is known that the equation \eqref{Chen-Dahmen} has no 
solutions for a family of $n$'s of natural density one. Moreover,  a Kraus type criterion is 
known which allows to check non-existence of solutions for all exponents $n$ up to 
$10^7$ (or more).  Let us also mention that  nonexistence of solutions for $n = 7$ 
follows from a more general result of Poonen, Schaefer and Stoll \cite{PSS}.

One of our motivations was to extend the above results (and methods) of Bruin, Chen 
and Dahmen, by considering some Diophantine equations $Ax^2 + By^{2n} = Cz^3$ 
with $(A,B,C)$'s different from $(1,1,1)$   
(assuming for simplicity that the class number of $\mathbb{Q}(\sqrt{-AB})$ is one).

Another motivation was to extend our previous results in the case $k = 3$. To be precise, 
given odd, coprime integers $a$, $b$ ($a>0$),  we 
consider the Diophantine equation 
\begin{equation}\label{maineq0} 
ax^2+b^{2n}=4y^k,   \quad  x, y\in \ZZ,\,\, n, k \in \NN, \, k \, odd \, \, prime,  \gcd(x,y)=1. 
\end{equation} 
The equation (\ref{maineq0}) was completely solved in \cite{DGS} 
for $a\in\{7,11,19,43,67,163\}$, and $b$ a power of an odd prime, 
under the conditions  $2^{k-1}b^n\not\equiv \pm 1\pmod a$ and $\gcd(k,b)=1$. 
The paper \cite{CHS} extends these results by removing the first of these 
assumptions, namely that $b$ is an odd prime. 
In this paper we fix $k=3$, but $b$ is arbitrary.

Let us briefly explain why we were unable to handle the Diophantine equations 
$7x^2 + y^{2n+1} = 4z^3$ and $x^2 + 7y^{2n+1} = 4z^3$. 
In \cite{PSS}, Poonen, Schaefer and Stoll find the primitive integer solutions to 
$x^2 + y^7 = z^3$.  Their method combine the modular method together 
with determination of rational points on certain genus-$3$ algebraic curves. 
This case (and possible generalizations to $Ax^2 + By^7 = Cz^3$) is  
very difficult (as explained by the authors in the Introduction to \cite{PSS}). 
Freitas, Naskrecki, Stoll \cite{FNS} considered a general Diophantine equation $x^2 + y^p = z^3$  
(with $p$ any prime $> 7$).  They follow and refine the arguments of \cite{PSS} by 
combining new ideas around the modular method with recent approaches 
to determination of the set of rational points on certain algebraic 
curves. As a result, they were able to find (under GRH) the complete 
set of solutions of the Diophantine equation $x^2 + y^p = z^3$ 
only for $p=11$.

\bigskip 

It is the aim of this paper (and the next one \cite{ChDS}) to consider the Diophantine equations  
\begin{equation}\label{maineq} 
ax^2+y^{2n}=4z^3,    \quad  x, y, z\in \ZZ,\,\,      \gcd(x,y)=1, \,\,   n \in \NN_{\geq 2}, 
\end{equation} 
and 
\begin{equation}\label{maineq2} 
x^2+ay^{2n}=4z^3,    \quad  x, y, z\in \ZZ,\,\,      \gcd(x,y)=1, \,\,   n \in \NN_{\geq 2}, 
\end{equation} 
where the class number of $\mathbb{Q}(\sqrt{-a})$  with $a\in\{7,11,19,43,67,163\}$ is 1.
An easy observation (see Corollary \ref{trivial}) is that we may only have solutions for $a=7$. 
Hence, below we will treat in some detail the equation (\ref{maineq}) for $a=7$ (the equation 
(\ref{maineq2})  for $a=7$ will be treated in our next paper \cite{ChDS}). 

We show that the Diophantine equation $7x^2 + y^{2n} = 4z^3$ 
has no nontrivial solutions for $n =3,4,5$ (using Chabauty method 
for $n =5$ [Theorem \ref{xx}], 
and calculating the Mordell-Weil group of  the corresponding elliptic curve for $n = 3$ 
[Theorem \ref{2-6-3}]).  
In Section \ref{fixed p} we will describe a Kraus type criterion  (Theorem  \ref{KrausCrit2}),  
and its refinement  (Theorem \ref{KrausRefin})
for the first equation. The computational verification of the criteria for
 all primes $7 < p < 10^9$ gives the following result (Theorem \ref{2-2p-3 Kraus}):

 \begin{theorem}\label{thm1}
The Diophantine equation $7x^2+y^{2p}=4z^3$ has no primitive solutions 
for all primes $11\leq p < 10^9$, $p\not = 13$. 
\end{theorem}

Let us briefly explain why the cases $p = 7,13$ are omitted in the statement. It seems that the only available 
method (at present) for treating the Diophantine equation $7x^2+y^{14}=4z^3$ 
is to follow the methods of \cite{PSS} (the modular methods used in our article are not sufficient).  
On the other hand, it seems possible that further 
improvements of the Kraus type criterion (as in sections \ref{section on Kraus method}, \ref{refinement} and \ref{19}) 
may allow to treat the remaining case $p=13$, but we were unsuccessful.

We also use the symplectic method and quadratic reciprocity to show that the Diophantine 
equation $7x^2+y^{2p}=4z^3$ has no non-trivial proper solutions for a positive 
proportion of primes $p$ (Theorem \ref {2-2p-3 congr}): 

\begin{theorem}\label{thm2}
The Diophantine equation $7x^2+y^{2p}=4z^3$ has no primitive solutions 
for  a family of primes $p$ satisfying:
\[p\equiv 3 \text{ or } 55 \pmod {106}\quad \text{or}  \quad p\equiv 47,65,113,139,143 \text{ or } 167 \pmod {168}.
\]
\end{theorem}

\bigskip 

The article is structured as follows. 

In the preliminary Section \ref{sec.2} we show that, if $a\in\{11,19,43,67,163\}$, then 
the Diophantine equations  (\ref{maineq}) and (\ref{maineq2}) have no solutions. 

In Section \ref{sec.3} we will determine all families of solutions to the  
 Diophantine equation $7x^2+y^4=4z^3$ 
(variant of Zagier's result in the case $x^2+y^4=z^3$ \cite{Beu,Br}). 
Detailed proof of this result is given in the Appendix A. 

In Section \ref{sec.4} we completely solve the Diophantine equations $7x^2+y^6=4z^3$. 

In Section \ref{sec.5}  we completely solve the Diophantine equation $7x^2+y^8=4z^3$.

In Section \ref{modular approach} we use the modular approach to the Diophantine equation 
$7x^2+y^{2p}=4z^3$, with $p\geq 7$ a prime. The main results of this section are crucial for 
the proofs of Theorems \ref{thm1} and \ref{thm2}.
The Appendix B contains the proof of Lemma \ref{eliminating}.

In Section \ref{sec.7} we completely solve the Diophantine equation $7x^2+y^{10}=4z^3$, and comment on  the Diophantine equation $7x^2+y^{14}=4z^3$. 

In Section \ref{fixed p} we formulate a Kraus type criterion (actually, two criteria) for showing that this equation 
has no non-trivial proper integer solutions for specific primes $p > 7$. 
We computationally verify the criterion for all primes $7 < p < 10^9$, $p\not= 13$.  
The Appendix C (resp. D) contains proof of Corollary \ref{CorKrausCrit2}  (resp. \ref{corKrausRefin}).

In Section \ref{no solution for infinitely many p's} 
we use the symplectic method and quadratic reciprocity to show that the Diophantine 
equation $7x^2+y^{2p}=4z^3$ has no non-trivial proper solutions for a positive proportion of primes $p$.

\bigskip 

{\bf Acknowledgements.} We would like to thank Professor Frits Beukers for sending 
us his paper \cite{Beu},  and Professors Nils Bruin and Michael Stoll for providing us 
with some suggestions and references concerning Chabauty methods. 
We would also like to thank Professor Mike Bennett and Dr. Paul Voutier for useful 
comments concerning our results. 
We also cordially thank an anonymous referee for the comments and suggestions 
(the idea of splitting the original version into two papers) which allow to improve the final version. 
The third author was supported by the Research Fund of Bursa Uluda\u{g} University 
under Project No: F-2020/8.

\section{Preliminaries}\label{sec.2}  

\bigskip 

As the class number of $\mathbb Q(\sqrt{-a})$ with $a\in\{7,11,19,43,67,163\}$ is $1$, 
we have the following factorization for the left hand side of (\ref{maineq}) 

\[
\frac{y^n+x\sqrt{-a}} 2 \cdot \frac{y^n-x\sqrt{-a}} 2 = z^3. 
\]
Now we have 

\[
\frac{y^n+x\sqrt{-a}} 2 = \left(\frac{u+v\sqrt{-a}} 2\right)^3,  
\] 
where $u$, $v$ are odd rational integers. Note that necessarily $\gcd(u,v)=1$. 
Replacing $x$ with $y^n$, we obtain a similar factorization for the 
left hand side of (\ref{maineq2}). 
Equating the real and imaginary parts, we obtain the following result.

\begin{lemma}  \label{lem.1}
(a) Suppose that $(x,y,z)$ is a solution to (\ref{maineq}). Then 
\begin{equation} 
(x,y^n,z) = \left(\frac{v(3u^2-av^2)} 4, \frac{u(u^2-3av^2)} 4, \frac{u^2+av^2} 4 \right) 
\end{equation} 
for some odd $u,v\in\mathbb Z$ with $\gcd(u,v)=1$, $uv\not=0$. 

(b) Suppose that $(x,y,z)$ is a solution to (\ref{maineq2}). Then 
\begin{equation} 
(x,y^n,z) = \left(\frac{u(u^2-3av^2)} 4, \frac{v(3u^2-av^2)} 4, \frac{u^2+av^2} 4 \right) 
\end{equation} 
for some odd $u,v\in\mathbb Z$ with $\gcd(u,v)=1$, $uv\not=0$. 
\end{lemma}

\begin{corollary} \label {trivial} 
If $a\in\{11,19,43,67,163\}$, then the Diophantine equations  (\ref{maineq}) and 
(\ref{maineq2}) have no solutions. 
\end{corollary} 

\begin{proof}
We are reduced  to the equation $u(u^2-3av^2)=4y^n$ 
or  $v(3u^2-av^2)=4y^n$. Now, if 
$a\in\{11,19,43,67,163\}$, then $u(u^2-3av^2)$ is congruent to $0$ modulo $8$, 
while $4y^n$ is congruent to $4$ modulo $8$, a contradiction. Similarly in the 
second case. 
\end{proof}

\bigskip 

In what follows, we will consider the equations (\ref{maineq}) and (\ref{maineq2}) 
for $a=7$. From Lemma \ref{lem.1} we obtain infinitely many solutions with $n=1$. 
Below we will treat the cases $n\in\{2,3,4\}$ and $n\geq 5$ a prime, 
separately.

\section{The Diophantine equation $7x^2+y^4=4z^3$}\label{sec.3}

In this section we will determine all families of solutions to the title equation (variants of 
Zagier's result in the case $x^2+y^4=z^3$ \cite{Beu,Br}).

\begin{theorem} \label{2-4-3A} 

Let $x$, $y$, $z$ be coprime integers such that $7x^2+y^4=4z^3$. 
Then there are rational numbers $s$, $t$ such that one of the following holds.

\begin{equation}\label{family-(v)1} 
\begin{aligned}
&x = \pm  (1911s^4 + 1260ts^3 + 378t^2s^2 + 12t^3s + 7t^4) \\ 
& \qquad (-5078115s^8 - 11928168ts^7 - 2556036t^2s^6 - 1802808t^3s^5  \\ 
& \qquad    - 929922t^4s^4 - 38808t^5s^3 + 46620t^6s^2 + 9912t^7s + 461t^8), \\  
&y =  \pm 3  (21s^2-14ts-3t^2)  (2499s^4 + 1764ts^3 + 378t^2s^2 + 84t^3s - 5t^4),\\  
&z =  6828444s^8 + 7260624ts^7 + 6223392t^2s^6 + 1728720t^3s^5  \\  
&  \qquad  + 156408t^4s^4  + 49392t^5s^3 + 28224t^6s^2 + 3696t^7s + 268t^8, 
 \end{aligned} 
\end{equation}

\bigskip 

\begin{equation}\label{family-(v)2} 
\begin{aligned}
&x = \pm   (343s^4 - 84ts^3 + 378t^2s^2 - 180t^3s + 39t^4)  \\ 
&  \qquad  (1106861s^8 - 3399816ts^7 + 2284380t^2s^6 + 271656t^3s^5 \\ 
&  \qquad   - 929922t^4s^4 + 257544t^5s^3 - 52164t^6s^2 + 34776t^7s - 2115t^8), \\ 
&y = \pm 3  (21s^2-14ts-3t^2)  (-245s^4 - 588ts^3 + 378t^2s^2 - 252t^3s + 51t^4),\\  
&z =  643468s^8 - 1267728ts^7 + 1382976t^2s^6 - 345744t^3s^5 + 156408t^4s^4  \\  
&  \qquad  - 246960t^5s^3 + 127008t^6s^2 - 21168t^7s + 2844t^8.   
 \end{aligned} 
\end{equation}

\end{theorem}

For the proof of this result see the Appendix A.

\section{The Diophantine equation $7x^2+y^6=4z^3$}\label{sec.4}

\begin{theorem} \label{2-6-3} 
 Diophantine equation $7x^2+y^6=4z^3$ has no non-trivial solutions. 
\end{theorem}

\begin{proof}
This follows from the fact that the Mordell-Weil group of the elliptic curve given by $Y^2=X^3-2^4 7^3$ 
is trivial. 
\end{proof}

\section{The Diophantine equation $7x^2+y^8=4z^3$}\label{sec.5}

We will  prove variant of Bruin's result from section 5 in \cite{Br}.

Any primitive solution of the Diophantine equation $7x^2+y^8=4z^3$ satisfies, of 
course, the equation $7x^2+(y^2)^4=4z^3$.  Hence, using Theorem \ref{2-4-3A}, 
we obtain formulas describing $x$, $y^2$ and $z$.  In particular we have the following 
formulas for $y^2$:

\begin{equation*} 
\begin{aligned} 
&y^2 = \pm  3  (21s^2-14ts-3t^2)  (2499s^4 + 1764ts^3 + 378t^2s^2 + 84t^3s - 5t^4), \\ 
&y^2 = \pm 3 (21s^2-14ts-3t^2) (-245s^4 - 588ts^3 + 378t^2s^2 - 252t^3s + 51t^4).   
\end{aligned}
\end{equation*}  
Note that $t=0$ implies $y=0$. Therefore, nontrivial solutions correspond to affine 
rational points on one of the following genus two curves:

\begin{equation*} 
\begin{aligned} 
&\mathcal C_1:  Y^2 =  3  (21X^2-14X-3)  (2499X^4 + 1764X^3 + 378X^2 + 84X - 5),  \\  
&\mathcal C_2:   Y^2 = - 3  (21X^2-14X-3)  (2499X^4 + 1764X^3 + 378X^2 + 84X - 5),  \\ 
&\mathcal C_3:  Y^2 =  3  (21X^2-14X-3)  (-245X^4 - 588X^3 + 378X^2 - 252X + 51),   \\  
&\mathcal C_4:  Y^2 = - 3  (21X^2-14X-3)  (-245X^4 - 588X^3 + 378X^2 - 252X + 51).  
\end{aligned}
\end{equation*}

\begin{theorem} \label{2-8-3A-Theorem} 
The Diophantine equation $7x^2+y^8=4z^3$ has no non-trivial solutions.   
\end{theorem}

\begin{proof}  
We check that the  curves $\mathcal C_1$, $\mathcal C_2$, $\mathcal C_3$,  
and $\mathcal C_4$ have no affine rational points.  Indeed, the Magma command 
$\texttt{HasPointsEverywhereLocally}(f,2)$, 
gives 
 $\mathcal C_1(\mathbb Q_2) =  \mathcal  C_2(\mathbb Q_2) = 
\mathcal C_3(\mathbb Q_2) =  \mathcal  C_4(\mathbb Q_2) 
= \emptyset$.  
\end{proof}

\section{A modular approach to the Diophantine equ\-ation  $7x^2+y^{2n}=4z^3$} 
\label{modular approach}

Now we will assume that $n=p$ is a prime $\geq 7$.  
By Lemma \ref{lem.1}(a) we have reduced the problem of solving the title equation to solving the equation 
$4y^p = u(u^2-21v^2)$ with odd $u$, $v$ and $y$. Since $\gcd(u,v)=1$,  
we have $d=\gcd(u,u^2-21v^2) | 3$.  
We can solve the equation corresponding to $d=1$ with $p\geq 11$ and $p\neq 19$ (Proposition \ref{prop. no solution if d=1}). In the case $d=3$ we obtain a partial result (Lemma \ref{eliminating}). 
We will continue with the modular approach in sections \ref{fixed p} and \ref{no solution for infinitely many p's}.

It is possible to use modular approach for $p=7$, i.e. we can associate a Frey type curve, use the modularity theorem and Ribet's level-lowering theorem. However, the methods of eliminating newforms that we use to prove our main results fail in this case 
(see the proof of Proposition \ref{prop. no solution if d=1} and Lemma \ref{eliminating}).  
For more remarks see Section \ref{sec.7.2}.

\subsection{Reducing to the case of signature $(2p,p,2)$}  \label{reducing}

Below we will  reduce the  problem to solving the equation 
$4y^p = u(u^2-21v^2)$ with odd $u$, $v$ and $y$, to the problem of solving equations of signature $(2p,p,2)$.

\bigskip 

(i) \underline{$d=1$}.  Writing $u=\alpha^p$, $u^2-21v^2=4\beta^p$, we 
arrive at 

\begin{equation} \label{BS1} 
\alpha^{2p} - 4\beta^p = 21v^2. 
\end{equation}  

\bigskip 

(ii) \underline{$d=3$}.  Since $v_3(u^2-21v^2)=1$ we have
\begin{equation*}
\begin{cases}
u=3^{p-1}\alpha^p \\
u^2-21v^2=12\beta^p, 
\end{cases}
\end{equation*}
with odd $\alpha$, $\beta$ satisfying $\gcd(\alpha,\beta)=1$. 
This leads to the equation 
\begin{equation} \label{Kraus2} 
3^{2p-3}\alpha^{2p}-4\beta^p=7v^2. 
\end{equation}

\subsection{The equation $\alpha^{2p} - 4\beta^p = 21v^2$}

First, let us give short elementary proof that the title equation (the equation (\ref{BS1})) has no solution 
for infinitely many primes $p$.  We will use the following result, which is a variant of Lemmas 
14 and 15 from \cite{Da}.

\begin{lemma} \label{lem.cong1} 
(a)  $\beta-\alpha^2$ is a square modulo $7$. 

(b)  $7 \nmid \alpha$ and $(\beta/\alpha^2)^p \equiv 2 \pmod  7$. 
\end{lemma}

\begin{proof}
(a) We have $\beta-\alpha^2 \mid 4(\beta^p-\alpha^{2p}) = -3(u^2+7v^2)$. 
Assume that $q$ is an odd prime dividing $\beta-\alpha^2$. Then 
$-(3u)^2 \equiv 7(3v)^2 \pmod q$, hence $q=3$ or $\left(\frac{-7 }{q}\right)=1$. 
Note that $v_3(\beta-\alpha^2)=1$, hence 
$\left(\frac{3^{v_3(\beta-\alpha^2)}}{ 7}\right) = \left(\frac 3 7\right) = -1$.  Moreover, 
$\left(\frac{-1}{ 7}\right)=-1$, $\left(\frac{2}{7}\right)=1$, 
$\left(\frac{q}{7}\right) = \left(\frac{-7}{ q}\right) =1$, and $\beta-\alpha^2 < 0$.  
Hence  $\left(\frac{\beta-\alpha^2}{7}\right) = 1$ as wanted. 

(b)  Note that $7 \mid \alpha$ implies $7 \mid \beta$, a contradiction.  Now 
$(\beta/\alpha^2)^p = (1-21(v/u)^2)/4 \equiv 2 \pmod 7$ as wanted.  
\end{proof}

\begin{proposition}  \label{prop.nosolution1} 
Let $p$ be a prime with $p\equiv 5 \pmod 6$. Then  (\ref{BS1}) has no non-trivial solutions. 
\end{proposition} 

\begin{proof}  From part (a) of Lemma \ref{lem.cong1} we get that $\beta/\alpha^2 - 1$ is a square 
modulo $7$, while from part (b) we get $\beta/\alpha^2 - 1 \equiv 3 \pmod 7$, a contradiction. 
\end{proof}

Now, let us use modular approach to prove a much stronger result.

\begin{proposition} \label{prop. no solution if d=1}
The Diophantine equation (\ref{BS1}) has no solutions in coprime odd integers for $p\geq 11$, 
$p\not = 19$.  
\end{proposition} 

\begin{proof}  We will apply the Bennett-Skinner strategy \cite{BS} to a more general equation 
$X^p-4Y^p=21Z^2$, $p\geq 7$.  We are in case (iii) of \cite[p.26]{BS},  hence from  
Lemma 3.2 it follows, that we need to consider the newforms of weight $2$ and levels 
$N\in\{1764, 3528\}$.

a) There are $13$ Galois conjugacy classes of forms of weight $2$ and level $1764$.  
We can compute systems of Hecke eigenvalues for conjugacy classes of newforms using 
Magma \cite{Magma} or use Stein's Modular Forms Database \cite{Stein}.  
We will use numbering as in Stein's tables. 

We can eliminate $f_{12}$, when $p\geq 11$, as follows. 
We have $c_5(f_{12})=\pm \sqrt{2}$  and so, by \cite[Prop. 4.3]{BS}, $p$ must divide one 
of $2$, $14$, $34$. On the other hand,  $c_{13}(f_{12})=\pm 3 \sqrt{2}$ and so $p$ must 
divide one of $2$, $14$, $18$. Similarly, we can eliminate $f_{13}$ when $p\geq 11$, 
considering $c_5(f_{13})$: in this case $p$ must divide one of $4$, $100$, $196$. 

Similarly, we can use \cite[Prop. 4.3]{BS} to eliminate $f_1$, $f_3$, $f_4$, $f_6$, $f_8$ and 
$f_{10}$. Elimination of $f_1$ and $f_{10}$:  we have $c_5(f_1) = -3$ and $c_5(f_{10})=3$, hence  
$p$ must divide one of $1$, $3$, $5$, $7$, $9$.  
Elimination of $f_3$ and $f_8$: we have $c_{13}(f_3)=3$ and $c_{13}(f_8)=-3$, hence  
$p$ must divide one of $1$, $3$, $5$, $7$, $9$, $11$, $17$; 
we have $c_{19}(f_3)=1$ and $c_{19}(f_8)=-1$, hence  $p$ must divide one of 
$1$, $3$, $5$, $7$, $9$, $19$, $21$. 
Elimination of $f_4$ and $f_6$: we have $c_{13}(f_4)=-5$ and $c_{13}(f_6)=5$, hence  $p$ must 
divide one of $1$, $3$, $5$, $7$, $9$, $11$, $19$;   we have $c_{19}(f_4)=1$ and $c_{19}(f_6)=-1$, 
hence  $p$ must divide one of $1$, $3$, $5$, $7$, $9$, $19$, $21$. 

The newforms $f_2$, $f_5$, $f_7$, $f_9$ and $f_{11}$ correspond to isogeny classes 
of elliptic curves $A$, $G$, $H$, $I$ and $K$ respectively (using notation from Cremona's online 
tables or from LMFDB)  \cite{Cremona}, 
with non-integral $j$-invariants $u/(3^v7^w)$, where $\gcd(u,21)=1$ and $v^2+w^2>0$. 
Hence we can eliminate all these forms using \cite[Prop. 4.4]{BS}.

b) There are $39$ Galois conjugacy classes of forms of weight $2$ and level $3528$.  
Again, we will use numbering as in Stein's tables. 

Let us consider $12$ classes with irrational Fourier coefficients first.  
To eliminate $f_{30}$, $f_{31}$, $f_{32}$, $f_{33}$, $f_{34}$, $f_{35}$, $f_{36}$, and 
$f_{36}$, it is enough to consider $c_5$. To eliminate the remaining $4$ newforms, we 
need to consider $c_5$ and $c_{11}$. 

Now let us consider the newforms with rational Fourier coefficients. 
The newforms $f_2$, $f_4$, $f_5$, $f_7$, $f_{13}$, $f_{14}$, $f_{20}$, 
$f_{21}$, $f_{22}$, $f_{23}$ and $f_{24}$ correspond to isogeny classes 
of elliptic curves $B$, $D$, $E$, $G$, $M$, $N$, $T$, $U$, $V$, $W$ 
and $X$ respectively, with non-integral $j$-invariants $u/(3^v7^w)$, 
where $\gcd(u,21)=1$ and $v^2+w^2>0$. 
Hence we can eliminate all these forms using \cite[Prop. 4.4]{BS}.  

To eliminate the remaining $16$ newforms, we will use  \cite[Prop. 4.3]{BS}. 
To eliminate $f_8$, $f_9$, $f_{10}$, $f_{11}$, $f_{12}$, $f_{15}$, $f_{16}$, 
$f_{17}$, $f_{18}$ and $f_{19}$, it is enough to consider $c_5$. 
To eliminate the remaining $6$ newforms, we consider $c_{13}$ 
(hence $p$ must divide one of $1$, $3$, $5$, $7$, $9$, $11$, $17$) and 
use Proposition \ref{prop.nosolution1}   to exclude $p\in \{11,17\}$. 
\end{proof}

\subsection{The equation $3^{2p-3}\alpha^{2p}-4\beta^p=7v^2$}

Let $(a,b,c)$ be a solution in coprime odd integers of the title equation (i.e. the equation \eqref{Kraus2}).
Following \cite{BS}, we consider the following Frey type curve associated to $(a,b,c)$
\begin{equation} \label{Frey2} 
E=E(a,b,c):  Y^2 = X^3 + 7cX^2 -7b^pX. 
\end{equation} 

We have $\Delta_E =  2^4\cdot 3^{2p-3}\cdot 7^3 (ab)^{2p}$ 
and $N_E = 588\cdot \prod_{l|ab}l$ 
(resp. $1176\cdot \prod_{l|ab}l$) if $b \equiv 3 \pmod 4$ 
(resp. $b \equiv 1 \pmod 4$).  Using \cite[Corollary 3.1]{BS}  we obtain,  
that the associated Galois representation 
$$
\overline{\rho}_{E,p}:  \Gal \to \operatorname{GL}_2(\mathbb F_p) 
$$
is irreducible for all primes $p\geq 7$. From \cite[Lemma 3.3]{BS} we know, that $\overline{\rho}_{E,p}$ arises from a cuspidal newform $f$ of weight 2, level $N=588$ (resp. $1176$), and trivial Nebentypus character. 
Applying \cite[Propositions 4.3 and 4.4]{BS} and \cite[Proposition 13]{FK} yields the following result.

\begin{lemma}\label{eliminating}
Let $p$ be a prime. Suppose that $(a,b,c)$ is a solution in coprime odd integers to the equation (\ref{Kraus2}). Let $E=E(a,b,c)$ be the associated Frey type curve. 
\begin{enumerate}
\item If $p\geq 13$, then $\overline{\rho}_{E,p}\cong \overline{\rho}_{F,p}$ for an elliptic curve $F$ in one of the following isogeny classes: $588C, 588E, 1176G, 1176H $.
\item If $p=11$, then  $\overline{\rho}_{E,p}\cong \overline{\rho}_{F,p}$ for an elliptic curve $F$ in one of the following isogeny classes: $588C, 588E, 1176A, 1176F, 1176G, 1176H$.
\end{enumerate}
\end{lemma}

We prove Lemma \ref{eliminating} in Appendix B. We use here Cremona labels (see \cite{Cremona}).

\section{The Diophantine equation $7x^2+y^{2n}=4z^3$ for $n = 5,7$}\label{sec.7}

\subsection{The Diophantine equation $7x^2+y^{10}=4z^3$}\label{sec.7.1}

\begin{theorem} \label{xx}
The Diophantine equation $7x^2+y^{10}=4z^3$ has no non-trivial solutions.   
\end{theorem} 

\begin{proof} 
We may consider the equations \eqref{BS1} and \eqref{Kraus2} for $p=5$, and in this case they 
lead to the curves 
$\mathcal C_1: Y^2 = 84X^5 + 21$ and $\mathcal C_2: Y^2  = 28X^5 + 3^7 \times 7$, respectively.  
Now $Jac(\mathcal C_i)$ ($i=1,2$) have $\mathbb Q$-rank $0$, and using 
$\texttt{Chabauty0}$, we obtain $\mathcal C_i(\mathbb Q) = \{\infty\}$ ($i=1,2$), 
and the assertion follows. 
\end{proof}

\subsection{The Diophantine equation $7x^2+y^{14}=4z^3$}\label{sec.7.2}

We expect that the title equation has no solution in coprime odd integers. 
However, as we wrote at the start of Section \ref{modular approach}, we have been
unable to do so. Here we discuss a few approaches to this equation and the
obstacles to making them work here.

\bigskip

(i) {\it The modular method}.  We may consider the equations \eqref{BS1} and \eqref{Kraus2} for $p=7$: 
$ X^{14} - 4Y^7 = 21Z^2$ and $3^{11}X^{14} - 4Y^7 = 7Z^2$, respectively. %Proof of Proposition 
%\ref{prop. no solution if d=1} is not enough to solve the first one equation, and Kraus type criterion from section  \ref{section on Kraus method}  is not enough to solve the second of the equations. 
In both cases, we could not exclude the possibility that the Galois representation associated to the Frey type curve arises from newform with nonrational Fourier coefficients (see the proof of Proposition 2 and the proof of Lemma \ref{eliminating} in Appendix B).

(ii) {\it Chabauty type approach in genus $3$}.  The Diophantine equations from (i) lead to the genus $3$ 
curves $\mathcal D_1: y^2 = x^7 +2^{12} \cdot 3^7 \cdot 7^7$ and 
 $\mathcal D_2: y^2 = x^7 +2^{12} \cdot 3^{11} \cdot 7^7$, respectively.  
Magma calculations show that the only rational points on $\mathcal D_i(\mathbb Q)$  (with bounds $10^9$) 
are points at infinity, as expected. Magma also 
shows that ranks of $Jac(\mathcal D_i)(\mathbb Q)$ ($i=1,2$) are bounded by $1$. There are two technical 
problems to use Chabauty method:  one needs explicit rational  points of infinite order (not easy to find) and there 
is no readily available implementation of Chabauty's method for (odd degree) hyperelliptic genus $3$ curves.  
Professor Stoll suggested to try the methods of his papers \cite{Stoll1,Stoll2}, but we 
were not able to follow his advise yet.

(iii) {\it Combination of the modular and Chabauty methods}.  One may consider a more general 
Diophantine equation $7x^2 + y^7 = 4z^3$, try to follow the paper by Poonen, Schaefer and Stoll 
\cite{PSS}, and then deduce the solutions for the original Diophantine equation. It seems 
a very difficult task, but maybe the only available way ...

\section{Solving the Diophantine equation $7x^2+y^{2p}=4z^3$ for a fixed prime $p\geq 11$} \label{fixed p}

In this and the next section we will continue the modular approach to the title equation that we started in Section \ref{modular approach}. 
Here we will assume that $p\geq 11$ is a prime and apply variants of the method introduced by Kraus in \cite{Kr}.
Kraus stated a very interesting criterion \cite[Th\'eor\`eme 3.1]{Kr} that often allows to prove that the Diophantine equation $x^3+y^3=z^p$ ($p$ an odd prime) has no primitive solutions for fixed $p$, and verified his criterion for all primes $17 \leq p < 10^4$. 
Such a criterion has been formulated (and refined) in other situations (see, for instance, \cite{Ch,ChS,DJK,Da}). 
In Subsection \ref{section on Kraus method} we will formulate such a criterion (Theorem \ref{KrausCrit2}) in the case of Diophantine equation \eqref{Kraus2}. 
In Subsection \ref{refinement} we will give a refined version of the criterion, and we will apply it to those values of $p$, 
for which Theorem \ref{KrausCrit2}  is not sufficient.
In Subsection \ref{19} we use Kraus method to the equation (\ref{BS1}) with the exponent $p=19$.
Magma calculations based on the corresponding algorithms  allow to state the following result.

\begin{theorem}\label{2-2p-3 Kraus} 
The Diophantine equation $7x^2+y^{2p}=4z^3$ has no primitive solutions 
for all primes $11\leq p < 10^9$, $p\not = 13$. 
\end{theorem} 

As was shown in section \ref{modular approach}, it is enough to deal with the  equation \eqref{Kraus2} with $p\geq 11$ and the equation (\ref{BS1}) with $p=19$ (see Proposition \ref{prop. no solution if d=1}). The proof of Theorem \ref{2-2p-3 Kraus} will take the whole  section and follow from Corollaries \ref{CorKrausCrit2} and \ref{corKrausRefin} and Proposition \ref{Kraus method for p=19}.

\subsection{Kraus type criterion} \label{section on Kraus method}

Let $q\geq 11$ be a prime number, and let $k\geq 1$ be an integer factor of $q-1$. 
Let $\mu_k(\mathbb F_q)$ denote the group of $k$-th roots of unity in $\mathbb F_q^{\times}$. 
Set 
$$
A_{k,q} := \{\xi\in\mu_k(\mathbb F_q):  \frac{1-2^23^3\xi}{ 3^3\cdot 7} \,\text{ is  a  square  in} \, \mathbb F_q \}. 
$$
For each $\xi\in A_{k,q}$, we denote by $\delta_{\xi}$ the least non-negative integer such 
that 
$$ 
\delta_{\xi}^2 \, \mod q = \frac{1-2^23^3\xi}{ 3^3\cdot 7}. 
$$ 
We associate with each $\xi\in A_{k,q}$ the following equation 
$$ 
Y^2 = X^3 + 7\delta_{\xi}X^2 - 7\xi X. 
$$ 
Its discriminant equals $2^4 3^{-3} 7^3 \xi^2$, so it defines an elliptic curve $E_{\xi}$ over 
$\mathbb F_q$. We put $a_q(\xi) :=  q+1-\#E_{\xi}(\mathbb F_q)$.

\renewcommand{\labelenumi}{(\arabic{enumi})}

\begin{theorem}  \label{KrausCrit2} 
Let $p\geq 13$ be a prime (resp. $p=11$). Suppose that for each elliptic curve \[F\in \{588C1,1176G1\}\qquad (\text{resp. } F\in \{588C1,1176A1, 1176G1\} )\]
there exists a positive integer $k$ such that the following three conditions 
hold 
\begin{enumerate}
\item $q:=kp+1$ is a prime,
\item $a_q(F)^2 \not\equiv 4 \pmod p$, 
\item $a_q(F)^2 \not\equiv a_q(\xi)^2 \pmod p$ for all $\xi \in A_{k,q}$. 
\end{enumerate}  
Then the equation $3^{2p-3}x^{2p}-4y^p=7z^2$ has no solutions in coprime odd integers. 
\end{theorem} 

\begin{proof}
Let $p\geq 11$ be a prime. Suppose that $(a,b,c)$ is a solution of the equation (\ref{Kraus2}), where $a$, $b$, $c$ are coprime odd integers. Let $E$ denote the Frey type curve associated to $(a,b,c)$.
From Lemma \ref{eliminating} it follows that $\ov\rho_{E,p}\cong \ov\rho_{F,p}$, where $F$ is one of the curves:
\begin{alignat*}{2}
588C1, 588E1, 1176G1, 1176E1 \quad & \text{if }  p\geq 13,\\
588C1, 588E1, 1176A1, 1176F1, 1176G1, 1176H1 \quad & \text{if }  p=11.
\end{alignat*}
We have $a_l(E)\equiv a_l(F)\pmod p$ for all primes $l$ such that $l\nmid N_E$ and $a_l(F)\equiv \pm(l+1)\pmod p$ for all primes $l$ such that $l\mid N_E$ but $l\nmid N_F$.

Suppose $k$ is an integer that satisfies conditions (1) -- (3). If $q\mid N_E$, then $a_q(F)\equiv\pm (q+1)\equiv \pm 2\pmod p$, a contradiction. So the curve $E$ has good reduction at $q$ and  $a_q(F)\equiv a_q(E)\pmod p$.  From (\ref{Kraus2}) it follows that 
\[ \left(\frac{c}{(3a)^p}\right)^2=\frac{1-2^2 3^3 (\frac{b}{9a^2})^p}{3^3\cdot 7}.\]
Let $\xi= (\ov{b}/9\ov{a}^2)^p\in A_{k,q}$, where $\ov a$ and $\ov b$ are reductions of $a$ and $b$ modulo $q$. The reduction of $E$ modulo $q$ is a quadratic twist of $E_\xi$ by $(3\ov a)^p$ or $-(3\ov a)^p$. Hence $a_q(E)^2=a_q(\xi)^2$ and so $a_q(F)^2 \equiv a_q(\xi)^2\pmod p$, a contradiction.

The elliptic curves $588C1$ and $588E1$ have the same $j$-invariant, which is different from 0 and 1728, so they 
are isomorphic over some quadratic extension of $\QQ$ and  $a_l(588C1)^2=a_l(588E1)^2$ for every prime $l$. 
Hence, it suffices to consider only the first of the two curves. A similar remark applies to the other two pairs of curves: 
$1176A1$, $1176F1$ and $1176G1$, $1176H1$. 
\end{proof}

\begin{corollary} \label{CorKrausCrit2} 
Let $11\leq p < 10^9$ and $p\not = 13, 17$ be a prime. Then there are no triples 
$(x,y,z)$ of coprime odd integers satisfying $3^{2p-3}x^{2p} - 4y^p = 7z^2$. 
\end{corollary}

The proof of Corollary \ref{CorKrausCrit2} is based on computations in Magma and is contained in the Appendix C. We applied Theorem \ref{KrausCrit2} for all primes in range $11\leq p < 10^9$. We could not find an integer $k$ satisfying the conditions 
(1) -- (3) in the following two cases
$(F,p) = (588C1, 13)$ and $(F,p) = (1176G1, 17)$. We found such $k$ in any other case. We will use this information later 
in Section \ref{refinement}.

\subsection{A refined version of Kraus type criterion} \label{refinement}

As we have seen in Subsection \ref{section on Kraus method} Kraus criterion is not sufficient to prove that the equation (\ref{Kraus2}) with $p=13$ or $17$ has no solution in coprime odd integers. Following \cite{Da} we will refine the method and apply it successfully in case $p=17$.

Recall from subsection \ref{reducing} case (ii) that if  $3\mid y$ we have 
\[3u=(3\alpha)^p \quad \text{and} \quad u^2-21v^2=12 \beta^p\]
for some coprime odd integers $u$, $v$ and coprime odd integers $\alpha$, $\beta$ such that $y=3\alpha\beta$. Let $R=\ZZ[\omega]$, where $\omega=\frac{1+\sqrt{21}}{2}$, be the ring of integers of the number field $\QQ(\sqrt{21})$. Observe that $R$ has class number one. If we factor in $R$  the both sides of the second equation, then we obtain 
\[u+\sqrt{21}v=(3\pm\sqrt{21})x_1^p\varepsilon  \quad \text{and} \quad u-\sqrt{21}v=(3\mp\sqrt{21})x_2^p\varepsilon^{-1},\]
where $x_1,x_2\in R$ and $\varepsilon\in R^{\ast }$.

Suppose that $q=kp+1$ is a prime that splits in $R$. Let $\mathfrak{q}$ be a prime in $R$ lying above $q$. We have $R/\mathfrak{q}\simeq \FF_q$. Write $\ov x$ for the reduction of $x\in R$ modulo $\mathfrak{q}$ and write  $r_{21}$ for $\ov{\sqrt{21}}$.  From the above equalities it follows that for some $\xi_0,\xi_1, \xi_2\in\mu_k(\FF_q)$
\[3\ov{u}=\xi_0,\quad 
\ov{u}+r_{21}\ov{v}=(3\pm r_{21})\xi_1\ov{\varepsilon}  
\quad \text{and} \quad 
\ov{u}-r_{21}\ov{v}=(3\mp r_{21})\xi_2\ov{\varepsilon}^{-1}.
\]
If we divide the second and the third equality by $3\ov u $ we obtain 
\[
\frac{1}{3}+r_{21}\frac{\ov{v}}{3\ov{u}}=(3\pm r_{21})\xi_1'\ov{\varepsilon}  \quad \text{and} \quad 
\frac{1}{3}-r_{21}\frac{\ov{v}}{3\ov{u}}=(3\mp r_{21})\xi_2'\ov{\varepsilon}^{-1},
\]
where $\xi_1'=\xi_1/\xi_0$ and $\xi_2'=\xi_2/\xi_0$. 
Suppose further that $\ov{\varepsilon_f}\in\mu_k(\FF_q)$ for a fundamental unit $\varepsilon_f\in R^*$. Then also $\xi_1'\ov{\varepsilon}, \xi_2'\ov{\varepsilon}\in \mu_k(\FF_q)$. Hence $\frac{\ov{v}}{3\ov{u}}$ is an element of $S_{k,q}\cup S'_{k,q}$, where 
\[S_{k,q}=\left\{\delta\in\FF_q: \frac{1}{3+r_{21}}\left(\frac 13  +r_{21}\delta\right), \; \frac{1}{3-r_{21}}\left(\frac 13  -r_{21}\delta\right) \in \mu_k(\FF_q)
\right\},
\]
\[S'_{k,q}=\left\{\delta\in\FF_q: \frac{1}{3-r_{21}}\left(\frac 13  +r_{21}\delta\right), \; \frac{1}{3+r_{21}}\left(\frac 13  -r_{21}\delta\right) \in \mu_k(\FF_q)
\right\}.\]

For $\delta \in S_{k,q}\cup S'_{k,q}$ we define $\xi_\delta =\frac{1-3^3\cdot 7\delta^2}{3^32^2}$, which is an element of $\mu_k(\FF_q)$.
The equation 
\[Y^2 = X^3 + 7\delta X^2 - 7\xi_\delta X 
\]
defines an elliptic curve $E_\delta$ over $\FF_q$. We put $a_q(\delta) :=  q+1-\#E_{\delta}(\FF_q)$. 

The above consideration and the argumentation as the proof of Theorem \ref{KrausCrit2} imply the following result.

\begin{theorem}\label{KrausRefin}
Let $p> 11$ be a prime. Suppose that for each elliptic curve $F\in \{588C1,1176G1\}$ there exists a positive integer $k$ such that the following conditions hold 
\begin{enumerate}
\item $q:=kp+1$ is a prime,
\item $q$ splits in $\ZZ[\frac{1+\sqrt 21}{2}]$,
\item $q\mid \operatorname{Norm}_{\QQ(\sqrt{21})/\QQ}((\frac{5+\sqrt{21}}{2})^k-1)$,
\item $a_q(F)^2 \not\equiv 4 \pmod p$,
\item $a_q(F)^2 \not\equiv a_q(\delta)^2 \pmod p$ for all $\delta \in S_{k,q}\cup S'_{k,q}$.
\end{enumerate}
Then the equation $3^{2p-3}x^{2p}-4y^p=7z^2$ has no solutions in coprime odd integers. 
\end{theorem} 

Compared to Theorem \ref{KrausCrit2}, we have two additional conditions that must be satisfied by the integer $k$. However, we get more information about the hypothetical solution and so there are less congruences to check in the latter condition.
The set $S_{k,q}\cup S'_{k,q}$ has significantly fewer elements than the set $A_{k,q}$, which appears in Theorem \ref{KrausCrit2}. For example, if $(F,p)=(1176G1,17)$ and $k=374$ we have $\#(S_{k,q}\cup S'_{k,q})=18$ and $\#A_{k,q}=176$. In this case, conditions (1) -- (5) are satisfied. On the other hand, we could not find an integer $k$ satisfying conditions (1) -- (5) for $(F,p)=(588C1,17)$. Nevertheless, combining arguments of Theorem  \ref{KrausCrit2} and Theorem \ref{KrausRefin} allow us to prove the following result.

\begin{corollary}\label{corKrausRefin}
The equation $3^{31}x^{34}-4y^{34}=7z^2$ has no solutions in coprime odd integers.
\end{corollary}

See Appendix D for the proof.

Application of the refined Kraus method does not give any new information about the equation (\ref{Kraus2}) with $p=13$. If $F=588C1$,  we are unable to find an integer $k$ satisfying the conditions (1) -- (5) of Theorem \ref{KrausRefin} (see Appendix D). Thus, the only conclusion we can make  in case $p=13$ is that the equation (\ref{Kraus2}) has no solution in coprime odd integers $a,b,c$ with $b \equiv 3 \pmod 4$.

\subsection{No proper solution to the equation $7x^2+y^{38}=4z^3$}\label{19} 

To complete the poof of Theorem \ref{2-2p-3 Kraus}, we need to deal with the equation (\ref{BS1}) for $p=19$. We will treat this case with the same method as we treated the equation (\ref{Kraus2}) for $p\geq 11$ in Subsection \ref{section on Kraus method}.

\begin{proposition}\label{Kraus method for p=19}
The Diophantine equation $\alpha^{38}-4\beta^{19}=21v^2$ has no solution in coprime odd integers.
\end{proposition}

\begin{proof}
Let $(a,b,c)$ be a solution of the equation (\ref{BS1}), where $a,b,c$ are odd and coprime. Following \cite{BS}, we associate to $(a,b,c)$ the  following Frey type curve
\begin{equation} \label{Frey1} 
E=E(a,b,c):  Y^2 = X^3 + 21cX^2 -21b^pX. 
\end{equation}
For a prime $p\geq 7$, the Galois representation $\ov\rho_{E,p}$ arises from a newform $f\in S_2(N)$ with $N\in\{1764,3528\}$. If $p=19$ the only two newforms which cannot be eliminated by the methods applied in the proof of Proposition \ref{prop. no solution if d=1} are the newforms $f_4$ and $f_6$ of level $1764$. These newforms correspond to the isogeny classes of elliptic curves represented by 
\[F: \; Y^2=X^3-28\qquad \text{and}\qquad G:\; Y^2=X^3-259308\]
respectively. Observe that the two curves are isomorphic over $\QQ(\sqrt{21})$, hence $a_q(F)^2=a_q(G)^2$ for all primes $q$.

Let $q\geq 11$ be a prime number and let $k>1$ be an integer such that $k\mid q-1$. We define 
\[B_{k,q}=\{\xi\in\mu_k(\FF_q): \frac{1-4\xi}{21}\,\text{ is  a  square  in } \FF_q \}. 
\]
For each $\xi\in B_{k,q}$, we denote by $\delta_{\xi}$ the least non-negative integer such 
that 
\[
\delta_{\xi}^2 \, \mod q = \frac{1-4\xi}{ 21}. 
\] 
We associate with each $\xi\in B_{k,q}$ the following equation 
\[ 
Y^2 = X^3 + 21\delta_{\xi}X^2 - 21\xi X,
\] 
which defines an elliptic curve $E_\xi$ over $\FF_q$. We put $a_q(\xi) :=  q+1-\#E_{\xi}(\FF_q)$.

Proceeding as in the proof of Theorem \ref{KrausCrit2}, we obtain the following conclusion.
If there exists an integer $k$ such that $q=19k+1$ is a prime, $a_q(F)^2\not\equiv 4\pmod {19}$ and $a_q(F)^2\not\equiv a_q(\xi)^2\pmod {19}$ for all $\xi\in B_{k,q}$, then neither $\ov\rho_{E,19}\cong \ov\rho_{F,19}$ nor $\ov\rho_{E,19}\cong \ov\rho_{G,19}$, and the Proposition follows. It can be checked (e.g. in Magma) that the least such $k$ is equal $34$. 
\end{proof}

\section{No solutions to the Diophantine equation \\$7x^2+y^{2p}=4z^3$ for infinitely many prime $p$'s} 
\label{no solution for infinitely many p's}

In this section we will use ideas of the papers \cite{BCDY,Ch,ChS,Da,FK,KO} to prove the following result.

\begin{theorem} \label{2-2p-3 congr}
The Diophantine equation $7x^2+y^{2p}=4z^3$ has no primitive solutions 
for  a family of primes $p$ satisfying:
\[p\equiv 3 \text{ or } 55 \pmod {106}\quad \text{or}  \quad p\equiv 47,65,113,139,143 \text{ or } 167 \pmod {168}.
\]
\end{theorem} 

If the title equation has a primitive solution, then one of the equations (\ref{BS1}) or (\ref{Kraus2})  is solvable in coprime odd integers (see \ref{reducing}). 
Proposition \ref{prop. no solution if d=1} together with Proposition \ref{Kraus method for p=19}  say that  there is no such solution of  equation  (\ref{BS1}) for a prime exponent $\geq 11$. 
Now we would like to establish an infinite family of prime exponents for which the equation (\ref{Kraus2}) has no solution in coprime odd integers.

\subsection{Application of the symplectic method}\label{symplectic method}

Let $p\geq 3$ be a prime. Let $E$ and $E'$ be elliptic curves over $\QQ$ and write $E[p]$ and $E'[p]$ for their $p$-torsion modules. Write $\GQ$ for the absolute Galois group $\Gal$.
Let $\phi:E[p] \rightarrow E'[p]$ be a $\GQ$-modules isomorphism. There is an element $d(\phi)\in \FF_p^\times$ such that, for all $P,Q\in E[p]$, the Weil pairings satisfy $e_{E',p}(\phi(P),\phi(Q)) = e_{E,p}(P,Q)^{d(\phi)}$.
We say that $\phi$ is a {\it symplectic isomorphism}  if $d(\phi)$ is a
square modulo $p$ and an {\it anti-symplectic} otherwise. If the Galois representation $\ov {\rho}_{E,p}$ is irreducible then all $\GQ$-isomorphisms have the same symplectic type.

Write $\Delta$ and $\Delta'$ for minimal discriminants of $E$ and $E'$. Suppose $E$ and $E'$ have potentially good reduction at a prime $l$. Set  $\tilde\Delta=\Delta/l^{v_l(\Delta)}$ and $\tilde\Delta'=\Delta'/l^{v_l(\Delta')}$. Define a {\it semistability defect}  $e$ as the order of  the group $\operatorname{Gal}(\QQ^{un}_l(E[p])/\QQ^{un}_l)$. Define $e'$ in the same way. Note that if $E[p]\cong E'[p]$ then $e=e'$ \cite[Proposition 13]{FK}. If $l\geq 5$ then $e$ is the denominator of $v_l(\Delta)/12$ \cite{Kra90}. 

We apply the following criterion \cite[Theorem 5]{FK}.

\begin{lemma}\label{symplectic}
Let $p\geq 5$ and $l\equiv 3\pmod 4$ be prime numbers. Let $E$ and $E'$ be elliptic curves over $\QQ_l$ with potentially good reduction and $e=4$. Set
\[
r=
\begin{cases}
0 & \text{if } v_l(\Delta)\equiv v_l(\Delta')\pmod 4,\\
1 & \text{otherwise},
\end{cases}
\quad
t=
\begin{cases}
1 & \text{if } \left(\frac{\tilde\Delta}{l}\right) \left(\frac{\tilde\Delta'}{l}\right)=-1,\\
0 & \text{otherwise}.
\end{cases}
\]
Suppose that $E[p]$ and $E'[p]$ are isomorphic $\operatorname{G}_{\QQ_l}$-modules. Then 
\[  E[p]\text{ and }E'[p]\text{ are symplectically isomorphic } \Leftrightarrow 
\left(\frac{l}{p}\right)^r\left(\frac{2}{p}\right)^t=1.
\]
\end{lemma}

\begin{proposition} The Diophantine equation $3^{2p-3}X^{2p}-4Y^p=7Z^2$ has no solution in coprime odd integers 
for any prime $p\equiv 47, 65, 113, 139, 143$ or $167 \pmod {168}$.
\end{proposition}

\begin{proof} 
Let $p> 11$ be a prime and let $E$ be the Frey curve given by the equation (\ref{Frey2}). There exists an elliptic curve $F$ in the isogeny class with Cremona label $588C$, $588E$, $1176G$ or $1176H$, such that the $p$-torsion modules $E[p]$ and $F[p]$ are $\GQ$-isomorphic (see Lemma \ref{eliminating}). Since $\overline{\rho}_{E,p}$ is irreducible, this isomorphism is either symplectic or anti-symplectic.

Both curves $E$ and $F$ have potentially good reduction at 7 and their semistability defects are equal 4, so we may apply  Lemma \ref{symplectic}. If $F$ is isomorphic to $588C1$ or $1176G1$, we obtain that $E[p]$ and $F[p]$ are symplectically isomorphic. We assume that the same is true for other choices of $F$. This happens if and only if $(\frac{7}{p})=(\frac{2}{p})=1$. 

Now  we use \cite[Proposition 2]{KO} with $l=3$, which says that $E[p]$ and $F[p]$ are symplectically isomorphic if and only if 
\[\left(\frac{v_3(\Delta_E) v_3(\Delta_F)}p\right)=1.\]
To obtain a contradiction we assume that the last equality doesn't hold for any curve $F$ in the considered isogeny classes, i.e. $-3$ and $-6$ aren't squares modulo $p$. Summarizing, we have
\[\left(\frac{7}{p}\right)=\left(\frac{2}{p}\right)=1\quad \text{and}\quad \left(\frac{-3}{p}\right)=-1.\]
This is equivalent to the congruence condition stated in the Proposition.
\end{proof}

\subsection{Application of quadratic reciprocity}

For a given field $K$, let $(,)_K: K^{\times} \times K^{\times} \to \{\pm 1\}$ 
be the Hilbert symbol defined by

\begin{equation*} 
(A,B)_K =  
\begin{cases}
1  \quad \text{if} \, z^2=Ax^2+By^2 \, \text{has} \,  \text{a} \,  \text{nonzero} \, \text{solution} 
\, \text{in} \, K,\\
-1 \quad \text{otherwise}. 
\end{cases}
\end{equation*} 
Note that the Hilbert symbol is symmetric and multiplicative. We will let $(,)_q$, $(,)$ and $(,)_{\infty}$ 
to denote $(,)_{\mathbb Q_q}$, $(,)_{\mathbb Q}$ and $(,)_{\mathbb R}$, respectively.   
Let $A=q^{\alpha}u$, $B=q^{\beta}v$, with $u$, $v$ $q$-adic units. If $q$ is an odd prime, then 
$$
(A,B)_q = (-1)^{\alpha\beta\frac{q-1}{2}}\left(\frac{u}{ q}\right)^{\beta} \left(\frac{v}{ q}\right)^{\alpha}, 
$$ 
and 
$$
(A,B)_2 = (-1)^{\frac{u-1}{ 2}\frac{v-1}{ 2} + \alpha\frac{v^2-1}{8}+\beta\frac{u^2-1}{8}}. 
$$

For all nonzero rationals $a$ and $b$, we have 
\begin{equation} \label{quadraticreciprocity}
\prod_{q\leq \infty} (a,b)_q = 1 
\end{equation}
(Quadratic Reciprocity in terms of the Hilbert symbol). 
We will use the following result \cite[Proposition 15]{BCDY}.

\begin{lemma} \label{dioph eq and Hilbert symb}
Let $r$ and $s$ be nonzero rational numbers. Assume that the Diophantine equation 
$$ 
A^2 - r B^{2p} = s (C^p-  B^{2p}) 
$$
has a solution in coprime nonzero integers $A$, $B$ and $C$, with $BC$ odd. Then 
$$
(r,s(C-B^2))_2 \prod_{2 < q < \infty} (r,s(C-B^2))_{q} = 1, 
$$ 
where the product is over all odd primes $q$ such that $v_q(r)$ or $v_q(s)$ is odd. 
\end{lemma}

Combining modular method and Lemma  \ref{dioph eq and Hilbert symb} we prove the following result.

\begin{proposition}
The Diophantine equation $3^{2p-3}X^{2p}-4Y^p=7Z^2$ has no solution in coprime odd integers for any prime $p$ satisfying $p\equiv 3$ or $55\pmod {106}$.
\end{proposition}

\begin{proof}
We rewrite the equation as follows
\begin{equation}\label{ABC-equation}
7A^2-\frac{1}{27}B^{2p}=-4C^p.
\end{equation}
By adding $4B^{2p}$ to both sides of (\ref{ABC-equation}) and dividing  by 7 we obtain
\[A^2+\frac{107}{189}B^{2p}=-\frac{4}{7}\left(C^p-B^{2p}\right).\]
From Lemma \ref{dioph eq and Hilbert symb} and the definition of the Hilbert symbol we have
\[
%\prod_{q\in\{2,3,7,107\}}\left(-\frac{749}{27}, -28(C-B^2)\right)_q=
\prod_{q\in\{2,3,7,107\}}\left(-3\cdot 7\cdot 107, -7(C-B^2)\right)_q=1.\]
The terms with $q\in\{2,3,7\}$ are equal 1 (for $q=7$ observe that  (\ref{ABC-equation}) with $p\equiv \pm 1\pmod 6$ imply that $\frac{C}{B^2}-1$ is a square modulo 7). 
Hence either $107\mid C-B^2$ or $(-7(C-B^2)/107)=1$. Using modular method we will derive a contradiction for $p\equiv 3$ or $55\pmod {106}$.

The Frey curve $E$ associated to $(A,B,C)$ is given by the equation 
\[y^2=x^3+7Ax^2-7C^px\] 
(this is the equation (\ref{Frey2}) with $(a,b,c)=(B/3,C,A)$). 
We know that $\ov\rho_{E,p}\cong \ov\rho_{F,p}$ for some elliptic curve $F$ in the isogeny class with Cremona label $588C$, $588E$, $1176G$ or $1176H$ (see Lemma \ref{eliminating}). We have $a_{107}(F)=-14$ or $8$. If $107\mid BC$ (i.e. $107 \parallel N_E$) then $a_{107}(F)\equiv \pm 108 \pmod p$. So in this case $p\in\{29,47,61\}$. But from Corollary \ref{CorKrausCrit2} it follows that for such $p$ the equation (\ref{ABC-equation}) has no solution in coprime odd integers. Hence $107\nmid BC$ and $a_{107}(E)\equiv a_{107}(F)\pmod p$. Since  $|a_{107}(E)-a_{107}(F)|<2\sqrt{107} +14<35$, then for $p> 35$ we have $a_{107}(E)=-14$ or $8$. 

If $107\mid C-B^2$, then $107\mid A$ and we obtain a contradiction since $a_{107}(E)=0$.
In case $107\nmid C-B^2$ the assumption $a_{107}(E)=-14$ or $8$ implies that 
\[\frac{\ov{C^p}}{\ov {B^{2p}}}\in S\subset\mathbb{F}_{107},  \quad\text{where}\; S=\{11,26,34,53,70,87,90,101\}.
\]
Here $\ov {C^p}$ and $\ov {B^{2p}}$ are reductions modulo 107 of $C^p$ and $B^{2p}$.
Next we check which exponents $p$ modulo $106$ have the property that $p$  is coprime to $106$ and for each $\zeta\in\mathbb{F}_{107}$ such that  $\zeta^p\in S$ the element $-7(\zeta -1)$ isn't a square in $\mathbb{F}_{107}$ (this contradicts the condition $(-7(C-B^2)/107)=1$). Such exponents $p$ satisfy $p\equiv 3$ or $55\pmod {106}$. 

To complete the proof we use Theorem \ref{2-6-3} (the case $p=3$).
\end{proof}

\noindent{\bf Remark.} 
If we add $-4B^{2p}$ instead of $4B^{2p}$ to both sides of (\ref{ABC-equation}), then by replacing $C$ with $-C$, we also may apply Lemma~\ref{dioph eq and Hilbert symb}. But combining this result with modular method as in the proof above gives no additional information on $p$.

\section*{Appendix A.  Proof of Theorem \ref{2-4-3A}}

\begin{proof} 
By Lemma \ref{lem.1}(a), 
we have reduced the problem to solving the equation 
$4y^2 = u(u^2-21v^2)$ with odd $u$, $v$ and $y$. Since $\gcd(u,v)=1$,  
we have $d=\gcd(u,u^2-21v^2) | 3$.  In this case, problem of solving the equation $7x^2+y^4=4z^3$ 
is reduced to solving the following equations 
\begin{equation}\label{neweq-1} 
 dX^4-(21/d)Y^2 = CZ^2,  
\end{equation}
where $d=1$ or $3$,  $C = \pm 1$, and $X$, $Y$ are odd, with $\gcd(X,Y)=1$. Below we will analyse all these  
cases in some detail.

$(i)$ If $(d,C)=(1,-1)$ or $(3,1)$, then \eqref{neweq-1} has no solutions. In these cases, we obtain a contradiction 
by reducing the equations modulo 3. 

$(ii)$ Now we are in the case $(d,C)=(1,1).$ Consider the Diophantine equation
\begin{equation}\label{eq.3.1.20}
Z^2+21Y^2=X^4.
\end{equation} 
So, we have $\gcd(X^2-Z,X^2+Z)=1$ or $2$. The first case gives
\begin{equation}\label{eq.3.1.22}
X^2=\dfrac{s^4+21t^4}{2},\,\,Z=\dfrac{s^4-21t^4}{2},\,\,Y=st,
\end{equation}
or 
\begin{equation}\label{eq.3.1.23}
X^2=\dfrac{3s^4+7t^4}{2},\,\,Z=\dfrac{3s^4-7t^4}{2},\,\,Y=st,
\end{equation}
with $(s,t)=1$, respectively.  Reducing the first equations of (\ref{eq.3.1.22}) and (\ref{eq.3.1.23}) modulo $3$ gives contradiction. 
If $\gcd(X^2-Z,X^2+Z)=2$, then $2 | Y$, a contradiction.

$(iii)$ Finally consider the case $(d,C)=(3,-1)$ for equation \eqref{neweq-1}.  
So, we have
\begin{equation}\label{eq.3.1.34}
-Z^2+7Y^2=3X^4.
\end{equation}
Put $K=X^2$. Then we obtain
\begin{equation}\label{eq.3.1.35}
Z^2+3K^2=7Y^2.
\end{equation}
Set $2Z\pm 3K=7L$ and $Z\mp 2K=7M$ with $\gcd(L,M)=1.$  
So, equation \eqref{eq.3.1.35} becomes
\begin{equation}\label{eq.3.1.36}
L^2+3M^2=Y^2
\end{equation}
and one gets
\begin{equation}\label{eq.3.1.37}
(K,Z)=(\pm (L-2M),\pm (2L+3M)).
\end{equation} 
By \eqref{eq.3.1.36}, we have $\gcd(Y+L,Y-L)=1$ or $2$.  
Now, $L$ is odd (otherwise $K=X^2$ is even, a contradiction), hence necessarily $\gcd(Y+L,Y-L)=2$. 
In this case we have 

\begin{equation}\label{eq.3.1.40}
\begin{cases}
Y\pm L=2^{2a-1}\alpha^2 \\
Y\mp L=2\cdot 3\beta^2
\end{cases}
\text{or} \,\,\,
\begin{cases}
Y\pm L=3\cdot 2^{2a-1}\alpha^2 \\
Y\mp L=2\beta^2
\end{cases}
\end{equation}
with odd $\alpha$, $\beta$ satisfying  $\gcd(\alpha,\beta)=1$ and $a\ge 1$.  It follows that 
\begin{equation}\label{eq.3.1.41}
Y=2^{2a-2}\alpha^2+3\beta^2,\,\,L=2^{2a-2}\alpha^2-3\beta^2,\,\,M=2^a\alpha\beta 
\end{equation}
or 
\begin{equation}\label{eq.3.1.42}
Y=3\cdot 2^{2a-2}\alpha^2+\beta^2,\,\,L=3\cdot 2^{2a-2}\alpha^2-\beta^2,\,\,M=2^a\alpha\beta, 
\end{equation}
respectively. By \eqref{eq.3.1.35}, \eqref{eq.3.1.37}, \eqref{eq.3.1.41} and \eqref{eq.3.1.42},  
all solutions of \eqref{eq.3.1.34} are given by 
\begin{equation}\label{eq.3.1.43}
\begin{aligned}
&X^2=\pm (2^{2a-2}\alpha^2-3\beta^2-2^{a+1}\alpha\beta),\\
&Y=2^{2a-2}\alpha^2+3\beta^2\\
&Z=\pm (2^{2a-1}\alpha^2-6\beta^2+3\cdot 2^a\alpha\beta) 
\end{aligned}
\end{equation}
or 
\begin{equation}\label{eq.3.1.44}
\begin{aligned}
&X^2=3\cdot 2^{2a-2}\alpha^2-\beta^2-2^{a+1}\alpha\beta,\\ 
&Y=3\cdot 2^{2a-2}\alpha^2+\beta^2,\\  
&Z=\pm (3\cdot 2^{2a-1}\alpha^2-2\beta^2+3\cdot 2^a\alpha\beta).  
\end{aligned}
\end{equation} 

If $a=1$ or $a\geq 3$, then we get a contradiction for both \eqref{eq.3.1.43} and \eqref{eq.3.1.44}.

If $a=2$, then  \eqref{eq.3.1.44} gives $X^2=12\alpha^2-\beta^2-8\alpha\beta$.  
Since $\alpha$, $\beta$ and $X$ are odd, we obtain $LHS \equiv 1 (\mod 8)$ and $RHS \equiv 3 (\mod 8)$, 
a contradiction. 
The remaining case is to consider \eqref{eq.3.1.43} with $a=2$ which corresponds the equations 
$X^2 + 4(\alpha-\beta)^2 = 7\beta^2$ and $X^2 + 7\beta^2 = 4(\alpha-\beta)^2$.  
The first one has no solutions: taking reduction modulo $8$ we obtain a contradiction.   
The second one has infinitely many solutions, giving two 2-parameter families of solutions 
of \eqref{eq.3.1.34}

\begin{equation}\label{eq.v-1} 
\begin{aligned}
&X =   \pm (21s^2-14ts-3t^2),\\  
&Y =  1911s^4 + 1260ts^3 + 378t^2s^2 + 12t^3s + 7t^4 ,\\  
&Z =  \pm (4998s^4 + 3528ts^3 + 756t^2s^2 + 168t^3s - 10t^4)     
\end{aligned}
\end{equation} 
and 
\begin{equation}\label{eq.v-2} 
\begin{aligned}
&X =   \pm (21s^2-14ts-3t^2),\\  
&Y =   343s^4 - 84ts^3 + 378t^2s^2 - 180t^3s + 39t^4,\\  
&Z =  \pm (-490s^4 - 1176ts^3 + 756t^2s^2 - 504t^3s + 102t^4).     
\end{aligned}
\end{equation} 

Now the families \eqref{family-(v)1} and \eqref{family-(v)2} follow immediately.    \end{proof}

\section*{Appendix B.  Proof of Lemma \ref{eliminating}} 
Let $p\geq 7$ be a prime. Suppose that $(a,b,c)$ is a solution in coprime odd integers to the equation (\ref{Kraus2}). Let $E=E(a,b,c)$ be the associated Frey type curve.  The Galois representation $\ov{\rho}_{E,p}$  arises from a cuspidal newform of level $N=588$ or $1176$, weight $2$ and trivial nebentypus character.

There are $6$ Galois conjugacy-classes of newforms in $S_2(588)$, and $15$ in $S_2(1176)$. 
Let $f_i\in S_2(588)$ ($i=1,...,6$) and $g_j\in S_2(1176)$ ($j=1,...,15$) denote the first newform in the $i$-th respectively $j$-th class. The numbering coincides with those in Magma.

We start elimination of the newforms by applying  \cite[Prop. 4.3]{BS} (and its improvement resulting from \cite[Prop. 3]{KO}). We do this with Magma using a function whose code is attached below. The function returns true if a given newform can be eliminated on base of this result (for all but finitely many primes $p$) and false otherwise. In the first case it also returns a finite set of primes $p> 7$ for which the elimination is not possible.

\begin{verbatim}
IsEliminable := function(newform)
  N := Level(newform);
  mu:= N;
  for k in [1..#PrimeDivisors(N)] do 
    mu *:= 1+1/PrimeDivisors(N)[k];
  end for;
  NormsDivisors := [ ];
  for l in [x: x in [1..Floor(mu/6)]  |  IsPrime(x) and 
                          GCD(N, x) eq 1 ] do
     cl := Coefficient(newform, l); 
     normProduct := Norm(cl-l-1) * Norm(cl+l+1);
     if Degree(newform) ne 1 then
       normProduct *:= l;
     end if;
     for r in [-Floor(Sqrt(l))..Floor(Sqrt(l))] do
       normProduct *:= Norm(cl - 2*r);
     end for;
     if normProduct ne 0 then
       Pl := PrimeDivisors(Integers() ! normProduct); 
       Append(~NormsDivisors, Pl);
     end if;
  end for; 
  if not IsEmpty(NormsDivisors) then
    Pf := Set(NormsDivisors[1]) diff {2, 3, 5};
    for k in [2..#NormsDivisors] do
      Pf meet:= Set(NormsDivisors[k]);
    end for;
    return true, Pf;
  else
    return false, _;
    end if;
end function;
\end{verbatim}

\begin{xrem}
The function \texttt{IsEliminable} eliminates each newform  that either has non-rational coefficients or corresponds to an isogeny class of elliptic curves over $\QQ$ with trivial 2-torsion. 
So the code may be applied to other Diophantine problems if an attached Frey type curve is defined over $\QQ$, has at least one rational point of order 2 and the corresponding level $N$ is small enough (see \cite{Magma}). 
\end{xrem}

Using the above code we eliminate the following newforms:
\begin{alignat*}{3}
f_1, f_4 & \in S_2(588) \quad && \text{for } p\geq 11,\\
g_2, g_4, g_{10}, g_{11}, g_{12}, g_{13}, g_{14}, g_{15} & \in S_2(1176) \quad && \text{for } p\geq 11,\\
g_1, g_6 &\in S_2(1176) \quad && \text{for } p\geq 13.
\end{alignat*}

The newforms $f_2,f_6 \in S_2(588)$ and $g_3, g_5 \in S_2(1176)$ correspond to the isogeny classes of elliptic curves with Cremona label $588B$, $588F$ and $1176C$, $1176E$ respectively. All curves in these classes have $j$-invariant, whose denominator is divisible by $7$. If $F$ is such a curve, then from \cite[Prop. 4.4]{BS} it follows that $\ov\rho_{E,p}\not\cong \ov\rho_{F,p}$ for $p\geq 11$.

The newform $g_9\in S_2(1176)$ corresponds to the isogeny class $1176I$. For each curve $F$ in this class %we have $v_7(\Delta_F)=6$, hence 
the semistability defect at 7 is equal 2, while the semistability defect of $E$ at 7 equals 4 (see Subsection \ref{symplectic method}).
Hence from \cite[Proposition 13]{FK} we have $\ov\rho_{E,p}\not\cong \ov\rho_{F,p}$ for $p\geq 11$.

If $p\geq 13$, there are four newforms to eliminate left: $f_3,f_5\in S_2(588)$ and $g_7, g_8 \in S_2(1176)$ or equivalently four isogeny classes of elliptic curves: $588C$, $588E$ and $1176G$, $1176H$. These curves are the main obstacle in modular approach to the equation (\ref{Kraus2}). In case $p=11$, there are  two more newforms which we couldn't eliminate, namely $g_1,g_6\in S_2(1176)$. These newforms correspond to the isogeny classes $1176A$ and $1176F$. This ends the proof of Lemma \ref{eliminating}.

\section*{Appendix C.  Proof of Corollary \ref{CorKrausCrit2}} 

To prove Corollary \ref{CorKrausCrit2} we perform calculations in Magma using the following function. 

\begin{verbatim}
AreCriterionConditionsMet := function(p,k,F)
  q := k*p+1;
  if not IsPrime(q) then
    return false;
  end if;
  frob := FrobeniusTraceDirect(F,q);
  if frob^2 mod p eq 4 then
    return false;
  end if;
  Fq := FiniteField(q);
  mu_k := AllRoots(One(Fq), k);
  for xi in mu_k do
    if IsSquare((1-108*xi)/189) then
      delta := AllRoots((1-108*xi)/189, 2)[1];
      E_xi := EllipticCurve([ Fq | 0, 7*delta, 0, -7*xi, 0]);
      frob_xi := q+1-#E_xi;
      if frob^2 mod p eq frob_xi^2 mod p then
        return false;
      end if;
    end if;
  end for;
  return true;
end function;
\end{verbatim}  

The function \texttt{AreCriterionConditionsMet} checks the conditions (1)--(3) of Theorem \ref{KrausCrit2} for a given prime $p$, integer $k$, and an elliptic curve $F$. In order to prove Corollary \ref{CorKrausCrit2} we had to consider the following curves:
\begin{alignat*}{2}
588C1, 1176G1 \quad & \text{for }  11<p<10^9,\\
588C1, 1176A1, 1176G1 \quad & \text{for }  p=11.
\end{alignat*}
For fixed $p$ and $F$ we looked for the least (even) $k\leq 500$ such that the above function returns true. In the table below we list the values of $k$ obtained for $p<60$.

\vspace{7pt}

\noindent\begin{tabular}{|l|c|c|c|c|c|c|c|c|c|c|c|c|c|}
\hline
$p$ & 11 & 13 & 17 & 19 & 23 & 29 & 31 & 37 & 41 & 43 & 47& 53 & 59\\
\hline
$k$ ($588C1$) & 6 & - & 6 & 10 & 2 &2 &12 &6 &18 &22 &6 &2 &12 \\
\hline
$k$ ($1176G1$) & 2 & 12 & - & 22 & 2 &2 &22 &16 &2 &4 &6 &2 &14 \\
\hline
\end{tabular}

\vspace{7pt}

\noindent Moreover, for $F=1176A1$ and $p=11$ we obtained $k=2$.

The largest value of $k$ found was  $372$ for the prime $p = 458121431$ (for both curves). Most of the values obtained are small, e.g. in both cases: $F=588C1$ and $F=1176G1$ we have $k\leq 16$  for more than half of the primes $p$ checked.

We found no such $k$ in the following cases:
\[(F,p)\in\{(588C1, 13), (1176G1, 17)\}.\]
Hence Corollary \ref{CorKrausCrit2} follows.

The computations took about 270 hours for each of the curves: $588C1$ and $1176G1$. For the calculations we used 
two desktop computers, each containing an Intel Pentium G4400 (3.3GHz) CPU 
and 8GB of RAM.

\section*{Appendix D.  Proof of Corollary \ref{corKrausRefin}}

Corollary \ref{corKrausRefin} follows from the following combination of Theorem \ref{KrausCrit2} and Theorem \ref{KrausRefin}. 
\vspace{5pt}

{\it Let $p>11$ be a prime. Suppose that for each $F\in \{588C1, 1176G1\}$ there exists an integer $k$ satisfying the conditions (1) -- (3) of Theorem  \ref{KrausCrit2} or the conditions (1) -- (5) of Theorem \ref{KrausRefin}. Then the equation (\ref{Kraus2}) has no solution in coprime odd integers.} 
\vspace{5pt}

For the computation based on Theorem \ref{KrausRefin} we use the following function. 
\begin{verbatim}
AreRefinedConditionsMet := function(p, k, F)
  q := k*p+1;
  if not IsPrime(q) then
    return false;
  end if;
  O<w> := IntegerRing(QuadraticField(21));
  if Norm((2+w)^k-1) mod q ne  0 then 
    return false;
  end if;
  if #Decomposition(O,q) ne 2 then
    return false;
  end if;
  frob := FrobeniusTraceDirect(F,q);
  if frob^2 mod p eq 4 then
    return false;
  end if;
  I := Decomposition(O,q)[1][1];
  Fq, h := ResidueClassField(O, I);
  mu_k := AllRoots(One(Fq), k);
  S1 := { d: d in Fq | 1/h(2+2*w)*(1/3+h(2*w-1)*d) in mu_k }
    meet { d: d in Fq | 1/h(4-2*w)*(1/3-h(2*w-1)*d) in mu_k };
  S2 := { d: d in Fq | 1/h(4-2*w)*(1/3+h(2*w-1)*d) in mu_k }
    meet { d: d in Fq | 1/h(2+2*w)*(1/3-h(2*w-1)*d) in mu_k };
  for delta in S1 join S2 do 
    xi := 1/(4*27)-7/4*delta^2;
    E_delta := EllipticCurve([ Fq | 0, 7*delta, 0, -7*xi, 0]);
    frob_delta := q+1-#E_delta;
    if frob^2 mod p eq frob_delta^2 mod p then
      return false;
    end if;
  end for;
  return true;
end function;
\end{verbatim}

For a given prime $p$, an integer $k$ and an elliptic curve $F$ the function \texttt{AreRefinedConditionsMet} checks if the conditions (1) -- (5) of Theorem \ref{KrausRefin} are satisfied. 
The function returns true for $p=17$, $F= 1176G1$ and $k=374$. 

Recall from Appendix C that if $p=17$ and $F= 588C1$ then $k=6$ satisfies conditions (1) -- (3) of Theorem \ref{KrausCrit2}. This completes the proof of Corollary \ref{corKrausRefin}.

If $p=13$ and $F=588C1$, the above function returns false for all $k\leq 10^5$. The computation took about 51 hours (Intel Pentium CPU G4400, 8GB RAM).

Karolina Cha\l upka, Institute of Mathematics, University of Szczecin, Wielkopolska 15, 
70-451 Szczecin, Poland,  

E-mail:   karolina.chalupka@usz.edu.pl

\bigskip 

Andrzej D\k{a}browski, Institute of Mathematics, University of Szczecin, Wielkopolska 15, 
70-451 Szczecin, Poland,  

 \begin{tabular}{rl}
E-mail: & dabrowskiandrzej7@gmail.com\\
and & andrzej.dabrowski@usz.edu.pl 
\end{tabular}

\bigskip 

G\"okhan Soydan, Department of Mathematics, Bursa Uluda\u{g} University, 16059 Bursa, Turkey,  

E-mail:  gsoydan@uludag.edu.tr

\end{document}